\documentclass[a4paper,12pt]{amsart}

\usepackage[english]{babel}
\usepackage{amsmath,amsthm,amssymb,amsfonts}
\usepackage[all]{xy}
\usepackage{multicol}
\usepackage{todonotes}
\usepackage[shortlabels]{enumitem} 
\usepackage[dvipsnames]{xcolor}
\usepackage{pdflscape}
\usepackage{amsthm}
\usepackage{mathtools}
\usepackage{extarrows}

\usepackage{thmtools}

\usepackage[export]{adjustbox}
\usepackage{graphicx}
\usepackage{subcaption}
\usepackage{tikz}
\usetikzlibrary{mindmap,backgrounds, calc}
\usetikzlibrary{matrix,chains,scopes,positioning,arrows,fit}
\usetikzlibrary{positioning, shapes, patterns}
\usetikzlibrary{automata} 
\usetikzlibrary{shapes.geometric, arrows, arrows.meta}
\usetikzlibrary{positioning,decorations.markings}
\usepackage{caption}
\usepackage[normalem]{ulem}
\usepackage{marginnote}
\usepackage{cancel}
\usepackage{bbm}
\usepackage{comment}

\setlist[enumerate]{itemsep=0.2em, topsep=0.25em}

\makeatletter
\newcommand*\bigcdot{\mathpalette\bigcdot@{.5}}
\newcommand*\bigcdot@[2]{\mathbin{\vcenter{\hbox{\scalebox{#2}{$\m@th#1\bullet$}}}}}
\makeatother

\newcommand{\vertiii}[1]{{\left\vert\kern-0.25ex\left\vert\kern-0.25ex\left\vert #1 
		\right\vert\kern-0.25ex\right\vert\kern-0.25ex\right\vert}}

\makeatletter
\renewcommand{\tocsection}[3]{%
	\indentlabel{\@ifnotempty{#2}{\bfseries\ignorespaces#1 #2\quad}}\bfseries#3}
\renewcommand{\tocsubsection}[3]{%
	\indentlabel{\@ifnotempty{#2}{\ignorespaces#1 #2\quad}}#3}

\def\@tocline#1#2#3#4#5#6#7{\relax
	\ifnum #1>\c@tocdepth 
	\else
	\par \addpenalty\@secpenalty\addvspace{#2}%
	\begingroup \hyphenpenalty\@M
	\@ifempty{#4}{%
		\@tempdima\csname r@tocindent\number#1\endcsname\relax
	}{%
		\@tempdima#4\relax
	}%
	\parindent\z@ \leftskip#3\relax \advance\leftskip\@tempdima\relax
	\rightskip\@pnumwidth plus1em \parfillskip-\@pnumwidth
	#5\leavevmode\hskip-\@tempdima{#6}\nobreak
	\leaders\hbox{$\m@th\mkern \@dotsep mu\hbox{.}\mkern \@dotsep mu$}\hfill
	\nobreak
	\hbox to\@pnumwidth{\@tocpagenum{\ifnum#1=1\bfseries\fi#7}}\par
	\nobreak
	\endgroup
	\fi}
\AtBeginDocument{%
	\expandafter\renewcommand\csname r@tocindent0\endcsname{0pt}
}
\def\l@subsection{\@tocline{2}{0pt}{2.5pc}{5pc}{}}
\makeatother

\usepackage{tikz}
\usetikzlibrary{mindmap,backgrounds, calc}
\usetikzlibrary{matrix,chains,scopes,positioning,arrows,fit}
\usetikzlibrary{positioning, shapes, patterns}
\usetikzlibrary{automata} 
\usetikzlibrary{shapes.geometric, arrows, arrows.meta}
\usetikzlibrary{positioning,decorations.markings}

\usepackage[a4paper,left=3cm,right=3cm,top=2.5cm,bottom=3cm]{geometry}

\theoremstyle{definition}
\newcounter{maincoro}

\newtheorem*{theoremA}{Theorem A}

\newtheorem{theorem}{Theorem}[section]
\newtheorem{lemma}[theorem]{Lemma}

\newtheorem{corollary}[theorem]{Corollary}

\theoremstyle{definition}
\newcounter{maintheorem}

\theoremstyle{remark}

\numberwithin{equation}{section}

\newcommand{\R}{\mathbb{R}}
\newcommand{\C}{\mathbb{C}}

\newcommand{\T}{\mathbb{T}}
\newcommand{\D}{\mathbb{D}}

\makeatletter
\tikzset{
	arrow/.style   = {-{Implies}, double equal sign distance, line width=0.65pt},
	iffarrow/.style= {{Implies}-{Implies}, double equal sign distance, line width=0.65pt},
}

\renewcommand{\tocsection}[3]{%
	\indentlabel{\@ifnotempty{#2}{\bfseries\ignorespaces#1 #2\quad}}\bfseries#3}
\renewcommand{\tocsubsection}[3]{%
	\indentlabel{\@ifnotempty{#2}{\ignorespaces#1 #2\quad}}#3}

\def\@tocline#1#2#3#4#5#6#7{\relax
	\ifnum #1>\c@tocdepth 
	\else
	\par \addpenalty\@secpenalty\addvspace{#2}%
	\begingroup \hyphenpenalty\@M
	\@ifempty{#4}{%
		\@tempdima\csname r@tocindent\number#1\endcsname\relax
	}{%
		\@tempdima#4\relax
	}%
	\parindent\z@ \leftskip#3\relax \advance\leftskip\@tempdima\relax
	\rightskip\@pnumwidth plus1em \parfillskip-\@pnumwidth
	#5\leavevmode\hskip-\@tempdima{#6}\nobreak
	\leaders\hbox{$\m@th\mkern \@dotsep mu\hbox{.}\mkern \@dotsep mu$}\hfill
	\nobreak
	\hbox to\@pnumwidth{\@tocpagenum{\ifnum#1=1\bfseries\fi#7}}\par
	\nobreak
	\endgroup
	\fi}
\AtBeginDocument{%
	\expandafter\renewcommand\csname r@tocindent0\endcsname{0pt}
}
\def\l@subsection{\@tocline{2}{0pt}{2.5pc}{5pc}{}}
\makeatother

\DeclareMathOperator{\re}{Re}

\newcommand{\nn}[1]{{\left\vert\kern-0.25ex\left\vert\kern-0.25ex\left\vert #1 
		\right\vert\kern-0.25ex\right\vert\kern-0.25ex\right\vert}}

\renewcommand{\geq}{\geqslant}
\renewcommand{\leq}{\leqslant}

\newcommand{\NA}{\operatorname{NA}}

\newcommand{\e}{\varepsilon}

\newcommand{\supp}{\operatorname{supp}}

\newcommand{\eps}{\e}

\usepackage[bookmarks=true,hyperindex,pdftex,colorlinks,citecolor=cyan]{hyperref}
\hypersetup{colorlinks=true,  linkcolor=blue, citecolor=red, urlcolor=cyan}
\usepackage{cleveref}

\begin{document}

	\title[On the Bollobás theorem for complex $C(K)$-spaces]{On the Bollobás theorem for complex $C(K)$-spaces}
	
	\author[S.~Dantas]{Sheldon Dantas}
	\address[S.~Dantas]{Czech Technical University in Prague, FEE, Department of Mathematics, Technická 2, 16627, Prague 6, Czech Republic \newline
		\href{https://orcid.org/0000-0001-8117-3760}{ORCID: \texttt{0000-0001-8117-3760}}}
	\email{\texttt{sheldon.dantas@fel.cvut.cz}}
	\urladdr{sheldondantas.com}
	
	\author[H.~Del R\'{\i}o]{Helena del R\'{\i}o} 
	\address[H.~Del R\'{i}o]{Department of Mathematical Analysis and Institute of Mathematics (IMAG), University of Granada, E-18071 Granada, Spain \newline
		\href{https://orcid.org/0009-0004-5078-6993}{ORCID: \texttt{0009-0004-5078-6993} }}
	\email[]{\texttt{helenadelrio@ugr.es}}

	\begin{abstract}  Let $K$ and $S$ be compact Hausdorff spaces. We prove a Bollobás-type theorem for operators between complex spaces of continuous functions. More precisely, every operator that almost attains its norm at an initial function can be approximated by a norm-attaining operator whose norm-attaining function remains close to the original one. Our proof combines the iterative method developed previously in the real case with the recent construction for complex measures. In particular, our result provides a quantitative strengthening of the complex version of the Johnson-Wolfe density theorem. As a byproduct of our construction, the same conclusion can be obtained when considering only compact operators.
	\end{abstract}

	\thanks{ }
	
	\subjclass[2020]{Primary 46B20; Secondary 46E15, 47A30, 47B38}
	\keywords{Bishop--Phelps--Bollobás property; norm-attaining operators;
		spaces of continuous functions; complex Banach spaces}
	
	\maketitle
	
	
	\thispagestyle{plain}

	\section{Introduction}
	
	The Bishop-Phelps theorem states that the norm-attaining functionals on a Banach space are norm dense in its dual \cite{BishopPhelps}. Bollobás later obtained a quantitative refinement of this result: if a functional almost attains its norm at a given point, then both the functional and the point can be approximated by a norm-attaining pair \cite{Bollobas}. The corresponding problem for bounded linear operators led to the Bishop-Phelps-Bollobás property for operators, introduced
	and systematically studied by Acosta, Aron, García, and Maestre \cite{AAGM}. We refer to \cite{DGMR} for a general account of the subsequent development of the subject.
	
	A classical result of Johnson and Wolfe asserts that, whenever $K$ and $S$ are compact Hausdorff spaces, norm-attaining operators from the real space $C(K)$ into the real space $C(S)$ are dense in $\mathcal L(C(K),C(S))$ \cite{JohnsonWolfe}. This result was strengthened by Acosta, Becerra-Guerrero, Choi, Ciesielski, Kim, Lee, Lourenço and Martín, who proved that the pair $(C(K),C(S))$ has the Bishop-Phelps-Bollobás property in the real case \cite{ContinuousFunctions}. Their proof is based on the representation
	of operators by weak-star continuous families of measures and on an iterative perturbation procedure. However, several steps of the argument use the order structure of real measures and cannot be transferred
	directly to the complex setting.
	
	The corresponding complex density problem for $C(K)$-spaces remained open for several decades. It was just recently solved by García, Maestre and Rodríguez-Vidanes \cite{GMRV}. Their construction replaces the
	real-valued sign decomposition by a complex argument involving polar decompositions of complex measures, approximation by modulus one continuous functions and lower semicontinuity properties of the total variation. These tools make it possible to perturb a weak-star continuous family of complex measures while preserving the control needed in the iterative construction.
	
	The purpose of this paper is to combine this recent complex measure-theoretic construction with the Bishop-Phelps-Bollobás arguments developed in \cite{ContinuousFunctions}. We prove that, for arbitrary compact Hausdorff spaces $K$ and $S$, the pair $(C(K),C(S))$ has the Bishop-Phelps-Bollobás property for operators
	over the complex field. Thus, the real result from \cite{ContinuousFunctions} extends to
	complex spaces. Moreover, the norm-attaining density theorem of \cite{GMRV} is strengthened in the following quantitative sense: both the operator and a prescribed point at which it almost attains its norm can be simultaneously approximated by a norm-attaining pair. In other words, we prove the following result.
	
	Let $K$ and $S$ be compact Hausdorff spaces. Consider $C(K)$ and $C(S)$ as {\bf complex} Banach spaces with the supremum norm. The symbol $\NA(C(K), C(S))$ stands for the set of all bounded linear operators from $C(K)$ into $C(S)$ that are norm-attaining.
	
	\begin{theoremA} \label{theoremA} For every $0 < \eps < 1$, there exists $\eta(\eps) > 0$ such that whenever $T_0 \in S_{\mathcal{L}(C(K), C(S))}$ and $f_0 \in S_{C(K)}$ satisfy $\|T_0 f_0\| > 1 - \eta(\eps)$, there are $T \in S_{\mathcal{L}(C(K), C(S))}$ and $h \in S_{C(K)}$ such that 
		\begin{equation*}
			\|T(h)\| = 1, \ \ \ \|T - T_0\| < \eps \ \ \ \mbox{and} \ \ \ \|h - f_0\| < \eps.
		\end{equation*}
		In particular, we have that
		\begin{equation*}
			\overline{\NA(C(K), C(S))}^{\|\cdot\|_{\infty}} = \mathcal{L}(C(K), C(S)).
		\end{equation*}
	\end{theoremA}
	
	From our construction, we get that the same conclusion is obtained for compact operators. See Corollary \ref{compact-case} by the end of the manuscript.

	\section{Notation}
	
	Throughout the entire manuscript, $K$ and $S$ will be compact Hausdorff spaces. As we have mentioned in the introduction, the Banach spaces $C(K)$ and $C(S)$ are viewed as {\bf complex} spaces with the supremum norm. Let $\mathcal{M}(K):= C(K)^*$ be the space of complex regular Borel measures on $K$ with the total variation norm. If $\nu \in \mathcal{M}(K)$, then we write $d\nu = \theta d|\nu|$ for a polar decomposition, where $|\theta|=1$ $|\nu|$-almost everywhere. Recall that changing the values of $\theta$ on a $|\nu|$-null set does not alter the measure $\nu$ or any integral with respect to $\nu$. Every operator $T \in \mathcal{L}(C(K), C(S))$ has a unique representing map $\mu: S \rightarrow \mathcal{M}(K)$ which is $w^*$-continuous and satisfies 
	\begin{equation*}
		[Tf](s) = \int_K f d \mu(s) 
	\end{equation*}
	for every $f \in C(K)$ and $s \in S$. Moreover, $\|T\| = \|\mu\| = \sup_{s \in S} \|\mu(s)\|$. See, for instance, \cite[Theorem 1, page 490]{DunfordSchwartz} (or \cite[Lemma 2.1]{ContinuousFunctions}). See also \cite{Rudin} for a general background on spaces of continuous functions.
	
	Throughout this note, $B_X$ and $S_X$ denote the closed unit ball and the unit sphere of a Banach space $X$, respectively. We write $C(K)$, $C(K,[0,1])$ for the space of continuous functions from $K$ into $[0,1]$, $\mathcal{L}(C(K),C(K))$ the Banach space of bounded linear operators from $C(K)$ into itself and $\mathcal{K}(C(K),C(S))$ the closed subspace of all compact operators from $C(K)$ into $C(S)$.

	\section{Auxiliary lemmas}

	In this section, we collect all the lemmas we need to prove Theorem A. Although some parts of the proof closely follow the arguments of García, Maestre and Rodríguez-Vidanes, we include the complete proof for the reader’s convenience and to keep the paper self-contained as the entire proof is quite technical. Along the way, we indicate precisely the additional arguments required to obtain a Bollobás-type theorem, rather than merely the density of norm-attaining operators.

	The first lemma we need is almost identical as \cite[Lemma 1.1]{GMRV}. We present a proof for the sake of completeness.
	
	\begin{lemma} \label{lemma1} Let $K$ be a compact Hausdorff space and let $g \in C(K)$ be real-valued with $g \geq 0$. Then, 
		\begin{equation*}
			\nu \mapsto \int_K g d |\nu| 
		\end{equation*}
		is $w^*$-lower semicontinuous on $\mathcal{M}(K)$. Consequently, if $\mu: S \rightarrow \mathcal{M}(K)$ is $w^*$-continuous, then 
		\begin{equation*}
			s \in S \mapsto \int_K g d |\mu(s)| 
		\end{equation*}
		is lower semicontinuous with respect to the original topology of $S$. 
	\end{lemma}
	
	\begin{proof} Let $\Phi: \mathcal{M}(K) \rightarrow \R$ be defined by 
		\begin{equation*}
			\Phi(\nu):= \int_K g d |\nu|.
		\end{equation*}
		By \cite[Lemma 1.1]{GMRV}, we have that $\Phi$ is $w^*$-lower semicontinuous. We want to prove that $s \in S \mapsto \int_K g d |\mu(s)|$ is lower semicontinuous. This map is precisely $\Phi \circ \mu: S \rightarrow \R$. That is,
		\begin{equation*}
			(\Phi \circ \mu)(s) = \Phi(\mu(s)) = \int_K g d|\mu(s)| 
		\end{equation*}
		for every $s \in S$. Fix $\alpha \in \R$. To prove lower semicontinuity, it is enough to show that the set 
		\begin{equation*}
			\left\{ s \in S: \int_K g d |\mu(s)| > \alpha \right\} 
		\end{equation*}
		is open on $S$. Using the definition of $\Phi$, we have that 
		\begin{eqnarray*}
			\left\{ s \in S: \int_K g d |\mu(s)| > \alpha \right\} &=& \{ s \in S: \Phi(\mu(s)) > \alpha \} \\
			&=& \mu^{-1} \left( \{ \nu \in \mathcal{M}(K): \Phi(\nu) > \alpha \} \right).
		\end{eqnarray*}
		Since $\Phi$ is $w^*$-lower semicontinuous, the set $\{\nu \in \mathcal{M}(K): \Phi(\nu) > \alpha\}$ is $w^*$-open on $\mathcal{M}(K)$. Since $\mu: S \rightarrow \mathcal{M}(K)$ is $w^*$-continuous, its preimage its preimage is open in $S$. Hence, $s \mapsto \int_K gd|\mu(s)|$ is lower semicontinuous.
	\end{proof}
	
	The next lemma is the localization step in the proof. It is somewhat technical because, in addition to perturbing the representing map, we must preserve quantitative control of the prescribed function at which the original operator almost attains its norm. This additional control, together with the control of the representing map, is the ingredient needed to obtain a Bollobás theorem. The argument plays a role analogous to the combination of \cite[Lemmas~1.2 and~1.3]{GMRV}.
	
	\begin{lemma} \label{lemma2} Let $K$ and $S$ be compact Hausdorff spaces. Let $\mu: S \rightarrow \mathcal{M}(K)$ be $w^*$-continuous with $\|\mu\| = 1$. Let $f_0 \in S_{C(K)}$, let $s_0 \in S$ and suppose that 
		\begin{equation*}
			\re \int_K f_0 d\mu(s_0) > 1 - \alpha 
		\end{equation*}
		for some $\alpha > 0$. Let $0 < a < b < c < 1$, let $\tau > 0$ and set 
		\begin{equation*}
			m:= \frac{\alpha}{1 - c} + \tau \ \ \ \mbox{and} \ \ \ d:= (1 - a) + \alpha + m. 
		\end{equation*}
		Then, there exist $h \in S_{C(K)}$, a closed set $A \subseteq K$, a non-empty open set $U \subseteq S$ and a $w^*$-continuous map $\mu_1: S \rightarrow \mathcal{M}(K)$ such that 
		
		\begin{itemize}
			\item[(1)] $\|h - f_0\| \leq 1 - a$.
			\item[(2)] $|h(t)| = 1$ for every $ t \in K \setminus A$.
			\item[(3)] $K \setminus A \not= \emptyset$.
			\item[(4)] $|\mu_1(s)|(A) = 0$ for every $s \in U$.
			\item[(5)] $\|\mu_1\| \leq 1$. 
			\item[(6)] $\| \mu_1 - \mu\| < m$.
			\item[(7)] $\displaystyle \re \int_K h d \mu_1(s) > 1 - d$ for every $s \in U$.
		\end{itemize}
	\end{lemma}
	
	\begin{proof} Let us define 
		\begin{equation*}
			\mathcal{F}_b:= \{t \in K: |f_0(t)| \geq b \} \ \ \ \mbox{and} \ \ \ \mathcal{O}_a:=\{t \in K: |f_0(t)| > a\}.
		\end{equation*}
		Then, $\mathcal{F}_b$ is closed, $\mathcal{O}_a$ is open and $\mathcal{F}_b \subseteq \mathcal{O}_a$ since $a<b$. Since compact Hausdorff spaces are normal and, in particular, regular, there is an open set $\mathcal{V}$ such that $\mathcal{F}_b \subseteq \mathcal{V} \subseteq \overline{\mathcal{V}} \subseteq \mathcal{O}_a$. By Urysohn's lemma, applied to the disjoint closed sets $\mathcal{F}_b$ and $K \setminus \mathcal{V}$, there exists a real-valued function $u \in C(K, [0,1])$ such that $u \equiv1$ on $\mathcal{F}_b$ and $u \equiv 0$ on $K \setminus \mathcal{V}$. This implies that $\supp u \subseteq \overline{\mathcal{V}} \subseteq \mathcal{O}_a$.

		On $\mathcal{O}_a$, let us define $v:= \frac{f_0}{|f_0|}$. This function is continuous on $\mathcal{O}_a$. Since $\supp u \subseteq \mathcal{O}_a$, the product function $uv$, defined as 
		\begin{equation*}
			(uv)(t):=
			\begin{cases}
				u(t)v(t), & \text{if } t\in\mathcal{O}_a,\\
				0, & \text{if } t\in K\setminus\mathcal{O}_a,
			\end{cases}
		\end{equation*}
		is continuous on $K$. Now, let us define 
		\begin{equation*}
			h(t):=
			\begin{cases}
				f_0(t) + u(t) (v(t) - f_0(t)), & \text{if } t\in\mathcal{O}_a,\\
				f_0(t), & \text{if } t\in K\setminus\mathcal{O}_a.
			\end{cases}
		\end{equation*}
		Notice that $h$ is continuous on $K$. Also, since we can write $h(t) = (1 - u(t))f_0(t) + u(t) v(t)$ when $t \in \mathcal{O}_a$, $0 \leq u(t) \leq 1$, $|f_0(t)| \leq 1$ and $|v(t)| = 1$ for every $t \in \mathcal{O}_a$, we have that $\|h\| \leq 1$. Since $f_0 \in S_{C(K)}$, there exists $\tilde{t} \in K$ such that $|f_0(\tilde{t})| = 1$. Then, $\tilde{t} \in \mathcal{F}_b$ and so $u(\tilde{t})=1$, and $|h(\tilde{t})|=|v(\tilde{t})|=1$. In particular, $h \in S_{C(K)}$.

		Now, let 
		\begin{equation*} 
			A:= \{t \in K: |f_0(t)| \leq b\}. 
		\end{equation*} 
		Then, $A$ is closed and $K \setminus A$ is nonempty as $|f_0(\tilde{t})|=1 > b$. If $t \in K \setminus A$, then $|f_0(t)| > b$, so $t \in \mathcal{F}_b$, $u(t) = 1$ and $h(t) = v(t)$, that is, $|h(t)| = 1$. Moreover, for every $t \in \mathcal{O}_a$ (recall that outside $\mathcal{O}_a$, one has $h = f_0$), 
		\begin{eqnarray*}
			|h(t) - f_0(t)| &=& u(t) |v(t) - f_0(t)| \\
			&=& u(t) \left| \frac{f_0(t)}{|f_0(t)|} - f_0(t) \right| 
			= u(t)(1 -|f_0(t)|) \leq 1 - a
		\end{eqnarray*}
		as $u(t) > 0$ can occur only on $\mathcal{O}_a$, where $|f_0(t)| > a$. This proves items (1), (2) and (3).

		Next, let us choose a continuous function $G \in C(K, [0,1])$ such that $G \equiv 1$ on $A$ and $G \equiv 0$ on the set $\{t \in K: |f_0(t)| \geq c \}$ (this is once again Urysohn's lemma, since $A = \{t \in K: |f_0(t)| \leq b\}$ and $\{t \in K: |f_0(t)| \geq c\}$ are disjoint closed sets). Let $\mu: S \rightarrow \mathcal{M}(K)$ be $w^*$-continuous with $\|\mu\| = 1$. By assumption, we have that 
		\begin{equation*}
			\re \int_K f_0 d \mu(s_0) > 1 - \alpha 
		\end{equation*}
		for some $\alpha > 0$. Let us put $\nu_0:=\mu(s_0)$, a complex regular Borel measure on $K$. Since $\|\mu\| = 1$, we get that $\|\nu_0\| \leq 1$. Also, by above, we also get that 
		\begin{equation} \label{eq1}
			1 - \alpha < \re \int_K f_0 d \nu_0 \leq \int_K |f_0| d |\nu_0|.
		\end{equation}

		Let us write $\mathcal{W}_c:= \{t \in K: |f_0(t)| > c\}$. As $|f_0(t)| \leq c$ for every $t \in K \setminus \mathcal{W}_c$, $f_0 \in C(K)$ with $\|f_0\| = 1$ and since $|\nu_0|$ is the total variation of the complex measure $\nu_0$, it is a positive measure, and we can split the integral over $\mathcal{W}_c$ and $K \setminus \mathcal{W}_c$ as follows 
		\begin{eqnarray*}
			\int_K |f_0| d|\nu_0| &=& \int_{\mathcal{W}_c} |f_0| d |\nu_0| + \int_{K \setminus \mathcal{W}_c} |f_0| d |\nu_0| \\
			&\leq& |\nu_0|(\mathcal{W}_c) + c |\nu_0|(K \setminus \mathcal{W}_c) \\
			&=& |\nu_0|(\mathcal{W}_c) + c[|\nu_0|(K) - |\nu_0|(\mathcal{W}_c)] \\
			&=& c |\nu_0|(K) + (1 - c) |\nu_0|(\mathcal{W}_c) \\
			&=& c \|\nu_0\| + (1 - c) |\nu_0|(\mathcal{W}_c) \\
			&\leq& c + (1 - c)|\nu_0|(\mathcal{W}_c).
		\end{eqnarray*}
		Together with (\ref{eq1}), we obtain 
		\begin{equation*}
			|\nu_0|(\mathcal{W}_c) > 1 - \frac{\alpha}{1-c}.
		\end{equation*}
		Since $G \equiv 0$ on $\mathcal{W}_c$, we have that $1 - G \equiv 1$ on $\mathcal{W}_c$. Also, since $0 \leq G \leq 1$, $1-G \geq 0$ on $K$. Thus,
		\begin{eqnarray*}
			\int_K (1 - G)d|\nu_0| &=& \int_{\mathcal{W}_c} (1 -G)d|\nu_0| + \int_{K \setminus \mathcal{W}_c} (1 - G) d |\nu_0| \\
			&=& |\nu_0|(\mathcal{W}_c) + \int_{K \setminus \mathcal{W}_c} (1 - G)d|\nu_0| 
			\geq |\nu_0|(\mathcal{W}_c) 
			> 1 - \frac{\alpha}{1-c}.
		\end{eqnarray*}
		On the other hand, since $1- G \in C(K)$ satisfies $1 - G \geq 0$ on $K$ and $\mu: S \rightarrow \mathcal{M}(K)$ is $w^*$-continuous, it follows, from the second part of Lemma \ref{lemma1}, that $s \mapsto \int_K (1 - G)d|\mu(s)|$ is lower semicontinuous with respect to the original topology of $S$. From the previous inequality, there is an open neighbourhood $U_0$ of $s_0$ such that
		\begin{equation} \label{eq2}
			\int_K (1 - G)d |\mu(s)| > 1 - \frac{\alpha}{1-c} - \frac{\tau}{2}, \forall s \in U_0.
		\end{equation}
		Notice that since $\mu: S \rightarrow \mathcal{M}(K)$ is $w^*$-continuous, the scalar map $s \in S \mapsto \int_K f_0 d\mu(s)$ is continuous. By assumption, $\re \int_K f_0 d\mu(s_0) > 1 -\alpha$. This means that the set 
		\begin{equation*}
			\left\{ s \in S: \re \int_K f_0 d \mu(s) > 1 - \alpha \right\} 
		\end{equation*}
		is an open neighborhood of $s_0$ in $S$. We can then shrink $U_0$ if necessary such that 
		\begin{equation} \label{eq3}
			\re \int_K f_0 d \mu(s) > 1 - \alpha, \forall s \in U_0.
		\end{equation}
		From (\ref{eq2}) and the fact that $\|\mu(s)\| \leq 1$ for every $s \in U_0$, we obtain also that 
		\begin{equation} \label{eq4}
			\int_K G d |\mu(s)| < \frac{\alpha}{1 - c} + \frac{\tau}{2} < m, \forall s \in U_0 
		\end{equation}
		recalling the definition of $m$ in the statement of the present lemma.

		Now, let us choose a nonempty open neighbourhood $U$ with $s_0 \in U$ and $\overline{U} \subseteq U_0$ (recall that since $S$ is compact Hausdorff, it is normal and so it is regular). Once again by Urysohn's lemma on $S$, there exists $\psi \in C(S, [0,1])$ such that $\psi \equiv 1$ on $\overline{U}$ and $\supp \psi \subseteq U_0$. For every $s \in S$, we define 
		\begin{equation*}
			d \mu_1(s) := (1 - \psi(s)G)d \mu(s).
		\end{equation*}
		Then, for every $\varphi \in C(K)$, we have that 
		\begin{equation*}
			\int_K \varphi d\mu_1(s) = \int_K \varphi(1 - \psi(s)G(t)) d \mu(s)(t) 
		\end{equation*}
		for every fixed $s \in S$. To prove that $\mu_1: S \rightarrow \mathcal{M}(K)$ is $w^*$-continuous, we must prove that, for every $\varphi \in C(K)$, the scalar map $s \mapsto \int_K \varphi d\mu_1(s)$ is continuous. We have that 
		\begin{eqnarray*}
			\int_K \varphi d \mu_1(s) &=& \int_K \varphi(1 - \psi(s)G) d \mu(s) \\
			&=& \int_K \varphi d\mu(s) - \psi(s) \int_K \varphi G d \mu(s). 
		\end{eqnarray*}
		for a fixed $s \in S$. Now, $\varphi, G \in C(K)$ and so $\varphi G \in C(K)$. Since $\mu:S \rightarrow \mathcal{M}(K)$ is $w^*$-continuous, both maps $s \mapsto \int_K \varphi d \mu(s)$ and $s\mapsto \int_K \varphi G d \mu(s)$ are continuous. Also, $s \mapsto \psi(s)$ is continuous since $\psi \in C(K)$. Therefore, $s \mapsto \int_K \varphi d \mu_1(s)$ is continuous for every $\varphi \in C(K)$.

		Notice also that, since $0 \leq 1 - \psi(s) G \leq 1$ for every $s \in S$, we have that $\|\mu_1\| \leq 1$. If $s \in U \subseteq S$, then $\psi(s) = 1$ and $G \equiv 1$ on $A \subseteq K$. So, for every $s \in U$, we have that 
		\begin{equation*}
			|\mu_1(s)|(A) = \int_A (1 - \psi(s) G)d |\mu(s)| = 0.
		\end{equation*}
		In other words, this means that $\mu_1(s)$ gives no mass to $A$ in total variation for every $s \in U$. This yields items (4) and (5).

		Now, let us notice that if $s \not\in \supp \psi$, then $\mu_1(s) = \mu(s)$. On the other hand, if $s \in \supp \psi$, then $s \in U_0$, we get that 
		\begin{equation*}
			\|\mu_1(s) - \mu(s)\| = \int_K \psi(s) G d |\mu(s)| \leq \int_K G d |\mu(s)| \stackrel{(\ref{eq4})}{<} m.
		\end{equation*}
		Thus, $\|\mu_1 - \mu\| < m$ and this gives item (6).

		Finally, for $s \in U$, we have that $d \mu_1(s) = (1 - G) d \mu(s)$. Using that $\|\mu(s)\| \leq 1$ and that $\|h\| \leq 1$, we obtain, for every $s \in U$, 
		\begin{eqnarray*}
			\re \int_K h d \mu_1(s) &=& \re \int_K h (1 - G) d \mu(s) \\
			&\geq& \re \int_K f_0 d \mu(s) - \int_K |f_0 - h(1 - G)| d |\mu(s)| \\
			&\stackrel{(\ref{eq3})}{>}& 1 - \alpha - \int_K (|f_0 - h| + G|h|) d |\mu(s)| \\
			&\stackrel{(\ref{eq4})}{>}& 1- \alpha -(1-a) - m = 1 -d
		\end{eqnarray*}
		recalling how $d$ is defined in the statement above. This gives item (7).
	\end{proof}
	
	The next lemma is the main iterative step in the proof. Its purpose is to modify the representing map while keeping the function $h$ fixed and preserving the condition that the measures vanish on $A$. At the
	same time, the new map gives a better estimate for the integral of $h$. This is similar in spirit to the iterative argument used in \cite[Lemma~1.3]{GMRV}. However, here the function $h$ cannot be replaced during the construction, since it was chosen in Lemma \ref{lemma2} to remain close to the original function $f_0$.
	
	\begin{lemma} \label{lemma3} Let $K$ and $S$ be compact Hausdorff spaces. Let $\mu: S \rightarrow \mathcal{M}(K)$ be $w^*$-continuous. Let $U \subseteq S$ be nonempty and open. Let $A \subseteq K$ be closed with $K \setminus A \not= \emptyset$ and let $h \in C(K)$ satisfy $\|h\| = 1$ and $|h(t)| = 1$ for every $t\in K \setminus A$. Assume that $|\mu(s)|(A) = 0$ for every $s \in U$ and, for some $\delta > 0$, we also assume that 
		\begin{equation*}
			\re \int_K h d \mu(s) > \| \mu \| - \delta 
		\end{equation*}
		for every $s \in U$. Fix $1/2 < r < 1$. Then, there are a $w^*$-continuous map $\mu': S \rightarrow \mathcal{M}(K)$ and a nonempty open set $U' \subseteq U$ with the following properties. 
		\begin{itemize}
			\item[(1)] $|\mu'(s)|(A) = 0$ for every $s \in U'$.
			\item[(2)] $\|\mu'\| \leq \|\mu\|$.
			\item[(3)] $\displaystyle \re \int_K h d \mu'(s) > \|\mu'\| - r \delta$ for every $s \in U'$.
			\item[(4)] $\|\mu' - \mu\| \leq \sqrt{2 \|\mu\| \delta} + 2 \delta$.
		\end{itemize}
	\end{lemma}
	
	\begin{proof} If $\|\mu\| = 0$, then we take $\mu' = \mu$ and $U' = U$. Now, assume that $\|\mu\| >0$ and let us write $a_0:= 1 - r$. So, we have that $0 < a_0 < 1/2$. We distinguish two cases.
		
		\vspace{0.2cm}
		\noindent
		{\it Case 1}: Assume first that 
		\begin{equation*}
			\sup_{s \in U} \|\mu(s)\| \leq \|\mu\| - a_0 \delta.
		\end{equation*}
		Let $s_0 \in U$ and $t_0 \in K \setminus A$. By assumption, $|h(t_0)| = 1$. So, denoting by $\delta_{t_0}$ the Dirac measure on $K$ at $t_0$, we have that 
		\begin{equation*}
			\int_K h d(\overline{h(t_0)} \delta_{t_0}) = \overline{h(t_0)} h(t_0) = |h(t_0)|^2 =1.
		\end{equation*}
		Choose $\psi \in C(S, [0,1])$ with $\psi(s_0) = 1$ and $\supp \psi \subseteq U$. Let us define 
		\begin{equation*}
			\mu'(s) := \mu(s) + a_0 \delta \psi(s) \overline{h(t_0)} \delta_{t_0} 
		\end{equation*}
		for every $s \in S$. Notice that this map is $w^*$-continuous as, for every $\varphi \in C(K)$, we have that 
		\begin{equation*}
			\int_K \varphi d \mu'(s) = \int_K \varphi d \mu(s) + a_0 \delta \psi(s) \overline{h(t_0)} \varphi(t_0)
		\end{equation*}
		for a fixed $s \in S$. The first term is continuous since $\mu: S \rightarrow \mathcal{M}(K)$ is $w^*$-continuous and the second is also continuous since $\psi \in C(S)$. Hence $s \mapsto \int_K \varphi d \mu'(s)$ is continuous for every $\varphi \in C(K)$ and so $\mu'$ is $w^*$-continuous.

		Also, if $s \not\in U$, then $\mu'(s) = \mu(s)$ and if $s \in U$, then 
		\begin{equation*}
			\|\mu'(s)\| \leq \|\mu(s)\| + \|a_0 \delta \psi(s) \overline{h(t_0)} \delta_{t_0}\|.
		\end{equation*}
		Since $t_0 \in K \setminus A$, we have that $|\overline{h(t_0)}| = |h(t_0)| = 1$. Also, $\|\delta_{t_0}\| = 1$. Since $0 \leq \psi \leq 1$, we get that $\|\mu'(s)\| \leq \|\mu(s)\| + a_0 \delta$ for every $s \in U$. Since we are assuming $\sup_{s \in U} \|\mu(s)\| \leq \|\mu\| - a_0 \delta$, we get that 
		\begin{equation*}
			\|\mu'(s)\| \leq \|\mu\| - a_0 \delta + a_0 \delta = \|\mu\| 
		\end{equation*}
		for every $s \in U$. This shows that $\|\mu'\| = \sup_{s \in S} \|\mu'(s)\| \leq \|\mu\|$.

		Now, by assumption $|\mu(s)|(A) = 0$ for every $s \in U$. Since $t_0 \in K \setminus A$, we get that $|\mu'(s)|(A) = 0$ for every $s \in U$. Notice also that at $s_0 \in S$ we have 
		\begin{eqnarray*}
			\re \int_K h d \mu'(s_0) &=& \re \int_K h d \mu(s_0) + a_0 \delta \\
			&>& \|\mu\| - \delta + a_0 \delta \\
			&=& \|\mu\| - r \delta \\
			&\geq& \|\mu'\| - r\delta
		\end{eqnarray*}
		where the strict inequality is by assumption. By scalar continuity, the same strict inequality holds on some neighborhood $U' \subseteq U$ of $s_0$. Finally, 
		\begin{equation*}
			\|\mu' - \mu\| \leq a_0 \delta \leq \sqrt{2\|\mu\| \delta} + 2 \delta.
		\end{equation*}
		This proves items (1), (2), (3) and (4) in the first case. Now we consider the second.

		\vspace{0.2cm}
		\noindent
		{\it Case 2}: Let us assume now that 
		\begin{equation*}
			\sup_{s \in U} \|\mu(s)\| > \|\mu\| - a_0 \delta.
		\end{equation*}
		We can pick then $s_1 \in U$ to be such that 
		\begin{equation} \label{eq5}
			\|\mu(s_1)\| > \|\mu\| - a_0 \delta.
		\end{equation}
		Since $\mu(s_1) \in \mathcal{M}(K)$ is a complex regular Borel measure on $K$, we consider its total variation measure. Define $\sigma:= |\mu(s_1)|$. So, $\sigma$ is a positive finite regular Borel measure on $K$ and its total mass is $\sigma(K) = |\mu(s_1)|(K) = \|\mu(s_1)\|$. Now, the complex measure $\mu(s_1)$ is absolutely continuous with respect to $\sigma$ since if $\sigma(E) = |\mu(s_1)|(E) = 0$, then $\mu(s_1)(E) = 0$. Therefore, by the Radon-Nikodým theorem, there exists a $\sigma$-measurable function $\theta: K \rightarrow \C$ such that $d \mu(s_1) = \theta d \sigma$ with $|\theta|=1$ $\sigma$-almost everywhere.

		Now, since $s_1 \in U$, by assumption, $|\mu(s_1)|(A) = 0$ and so $\sigma(A) = 0$. Since $|h|=1$ on $K \setminus A$, the function $\xi = h \theta$ has complex modulus one $\sigma$-almost everywhere. The hypothesis implies, since $\|\mu(s_1)\| \leq \|\mu\|$, 
		\begin{equation} \label{eq6}
			\|\mu(s_1)\| - \re \int_K h d \mu(s_1) < \delta.
		\end{equation}
		Now, choose $\gamma > 0$ to be such that 
		\begin{equation} \label{eq7}
			\gamma < (2r - 1) \delta.
		\end{equation}
		Since $C(K)$ is dense in $L_1(\sigma)$, there exists $u \in C(K)$ such that 
		\begin{equation*}
			\int_K |u - \overline{\xi}|d \sigma < \gamma.
		\end{equation*}
		Define $P: \C \rightarrow \overline{\D}$ by 
		\begin{equation*}
			P(z):=
			\begin{cases}
				z, & \text{if } |z| \leq 1,\\
				z/|z|, & \text{if } |z| > 1.
			\end{cases}
		\end{equation*}
		The reason why we introduce the function $P$ is the following. The function $u \in C(K)$ above might not satisfy $\|u\| \leq 1$. So, we modify $u$ by projecting its values into the closed unit disk $\D$. That is, we consider the function $q:= P \circ u \in C(K)$. Then, $\|q\| \leq 1$. Now, since $\overline{\xi} \in \overline{\D}$ $\sigma$-almost everywhere, we have that 
		\begin{equation*}
			|P(u(t))-\overline{\xi(t)}| \leq |u(t)-\overline{\xi(t)}|.
		\end{equation*}
		for every $t \in K$ $\sigma$-almost everywhere. Thus,
		\begin{equation} \label{eq8}
			\int_K |q - \overline{\xi}| d \sigma \leq \int_K |u - \overline{\xi}| d \sigma < \gamma.
		\end{equation}
		Once again using that $|\xi| = 1$ $\sigma$-almost everywhere, we also get 
		\begin{equation} \label{eq9}
			\int_K |q \xi - 1| d \sigma = \int_K |(q - \overline{\xi}) \xi| d \sigma < \gamma
		\end{equation}
		by using (\ref{eq8}). Therefore, since $\sigma$ is a positive measure,
		\begin{eqnarray*}
			\re \int_K h q d \mu(s_1) &=& \int_K \re (qh \theta) d \sigma \\
			&=& \int_K \re (q \xi) d \sigma \\
			&\geq& \int_K1 d \sigma - \int_K |q - \overline{\xi}| d \sigma \\
			&=& |\mu(s_1)|(K) - \int_K |q - \overline{\xi}|d \sigma \\
			&=& \|\mu(s_1)\| - \int_K |q - \overline{\xi}|d \sigma \\
			&\stackrel{(\ref{eq9})}{>}& \|\mu(s_1)\| - \gamma \\ 
			&\stackrel{(\ref{eq5})}{>}& \|\mu\| - a_0 \delta - \gamma  \\
			&=& \|\mu\| - (1-r) \delta - \gamma \\
			&\stackrel{(\ref{eq7})}{>}& \|\mu\| - (1-r) \delta - (2r - 1) \delta \\
			&=& \|\mu\| - \delta + r \delta - 2r \delta + \delta \\
			&=& \|\mu\| - r \delta. \hspace{4.5cm} (\#_1)
		\end{eqnarray*}
		We need now to estimate $\int_K |q-1|d\sigma$. Since $|\xi| = 1$ $\sigma$-almost everywhere, we also have that $|\xi-1|^2=2(1-\re \xi)$ $\sigma$-almost everywhere. Using the Cauchy-Schwarz inequality in $L_2(\sigma)$, we get 
		\begin{eqnarray*}
			\int_K |\xi - 1| d \sigma &\leq& \left( \int_K |\xi - 1|^2 d \sigma \right)^{1/2} \sigma(K)^{1/2} \\
			&=& \left(2 \int_K (1 - \re \xi) d \sigma \right)^{1/2} \|\mu(s_1)\|^{1/2} \\
			&=& \left(2 \left[ \|\mu(s_1)\| - \re \int_K h d \mu (s_1) \right] \right)^{1/2} \|\mu(s_1)\|^{1/2} \\
			&\stackrel{(\ref{eq6})}{<}& \sqrt{2\|\mu\| \delta}.
		\end{eqnarray*}
		Putting this together with (\ref{eq8}) gives us the following inequality 
		\begin{equation} \label{eq10}
			\int_K |q-1|d \sigma \leq \int_K |q - \overline{\xi}|d\sigma + \int_K |\overline{\xi}-1| d \sigma \stackrel{(\ref{eq8})}{<} \gamma + \sqrt{2\|\mu\|\delta}.
		\end{equation}

		Let us set 
		\begin{equation*}
			\varphi:= \frac{|q-1|}{2} \in C(K).
		\end{equation*}
		Then, $0 \leq \varphi \leq 1$ and by Lemma \ref{lemma1} the map $s \mapsto \int_K(1-\varphi)d|\mu(s)|$ is lower semicontinuous since $\mu: S \rightarrow \mathcal{M}(K)$ is $w^*$-continuous. Let us choose $\eta > 0$ so small that 
		\begin{equation} \label{eq11}
			2 a_0 \delta + \gamma + 2 \eta < 2 \delta
		\end{equation}
		(which is possible provided we choose $\gamma$ small enough as before). There exists an open neighborhood $U_0 \subseteq U$ of $s_1$ such that, for every $s \in U_0$, we have 
		\begin{equation*}
			\int_K (1 - \varphi) d |\mu(s)| > \int_K (1 - \varphi) d \sigma - \eta.
		\end{equation*}
		Therefore, for every $s \in U_0$, we have 
		\begin{eqnarray*}
			\int_K \varphi d |\mu(s)| &=& \|\mu(s)\| - \int_K (1 - \varphi) d |\mu(s)| \\
			&\leq& \|\mu\| - \int_K (1 - \varphi) d |\mu(s)| \\
			&<& \|\mu\| - \int_K(1 - \varphi)d \sigma + \eta \\
			&=& \|\mu\| - \|\mu(s_1)\| + \int_K \varphi d \sigma + \eta.
		\end{eqnarray*}
		Hence, for $s \in U_0$, we obtain
		\begin{eqnarray*}
			\int_K |q -1| d |\mu(s)| &=& 2 \int_K \varphi d |\mu(s)| \\
			&<& 2 \left( \|\mu\| - \|\mu(s_1)\| + \int_K \varphi d \sigma + \eta \right) \\
			&\stackrel{(\ref{eq5})}{<}& 2 a_0 \delta + \int_K |q-1| d \sigma + 2 \eta \\
			&\stackrel{(\ref{eq10})}{<}& \sqrt{2 \|\mu\| \delta} + 2 a_0 \delta + \gamma + 2 \eta \\
			&\stackrel{(\ref{eq11})}{<}& \sqrt{2\|\mu\| \delta} + 2 \delta. \hspace{4.5cm} (\#_2)
		\end{eqnarray*}
		Once again, choose $\psi \in C(S, [0,1])$ to be such that $\psi(s_1)=1$ and $\supp \psi \subseteq U_0$. Define 
		\begin{equation*}
			d\mu'(s):=((1-\psi(s))+\psi(s)q) d\mu(s).
		\end{equation*}
		for every $s \in S$. Notice first that the coefficient $(1 - \psi(s)) + \psi q$ is a convex combination of $1$ and an element in $\D$. This means that we get $\|\mu'\| \leq \|\mu\|$. Now, let us see that $\mu': S \rightarrow \mathcal{M}(K)$ is $w^*$-continuous. For this, we must prove that, for every, $\phi \in C(K)$, the scalar function $s \mapsto \int_K \phi d \mu'(s)$ is continuous on $S$. Indeed, for every $\phi \in C(K)$, for a fixed $s \in S$, we have that 
		\begin{equation*}
			\int_K \phi d \mu'(s) = (1 - \psi(s)) \int_K \phi d \mu(s) + \psi(s) \int_K \phi q d \mu(s).
		\end{equation*}
		Since $s \mapsto \psi(s)$ is continuous since $\psi \in C(S)$, $s \mapsto \int_K \phi d \mu(s)$ is also continuous since $\mu: S \rightarrow \mathcal{M}(K)$ is $w^*$-continuous and $\phi \in C(K)$ and $q \in C(K)$, and again by the $w^*$-continuity of $\mu$, $s \mapsto \int_K \phi q d \mu(s)$ is continuous, we have that $s \mapsto \int_K \phi d \mu'(s)$ is continuous and $S$ and, therefore, $\mu': S \rightarrow \mathcal{M}(K)$ is $w^*$-continuous.

		Moreover, since $|\mu(s)|(A) = 0$ if $s \in U_0$, we have that $|\mu'(s)|(A) = 0$ for every $s \in U_0$. At $s_1 \in S$, we have that 
		\begin{equation*}
			\re \int_K h d \mu'(s_1) = \re \int_K h q d \mu(s_1) \stackrel{(\#_1)}{>} \|\mu\| - r \delta \geq \|\mu'\| - r \delta.
		\end{equation*}
		By scalar continuity, this strict inequality holds on a nonempty open set $U' \subseteq U_0$ containing $s_1$.

		Finally, for $s \in U_0$, since $d(\mu'(s) - \mu(s)) = \psi(s)(q-1) d \mu(s)$, we have that 
		\begin{equation*}
			\|\mu'(s) - \mu(s)\| \leq \int_K |q-1| d |\mu(s)| \stackrel{(\#_2)}{<} \sqrt{2\|\mu\| \delta} + 2 \delta
		\end{equation*}
		while $\mu'(s) = \mu(s)$ outside $\supp \psi \subseteq U_0$. Taking the supremum over $s \in S$, we get (4). This proves Case 2 and, consequently, the lemma.
	\end{proof}

	\section{Proof of Theorem A}
	
	We are ready now to prove Theorem A. The proof of it combines Lemma \ref{lemma2} and Lemma \ref{lemma3}. Starting from an operator which almost attains its norm at a function $f_0$, Lemma \ref{lemma2} provides a nearby function $h$ and a first perturbation of the representing map such that $h$ has modulus one on the part of $K$ where the relevant measures are concentrated. Lemma \ref{lemma3} can then be applied repeatedly, while keeping both $h$ and this concentration property unchanged, to obtain successively
	better estimates for the integral of $h$. The resulting sequence of representing maps converges to a $w^*$-continuous map which attains its norm at $h$. Since $h$ remains close to $f_0$ and the successive perturbations are quantitatively controlled, the corresponding operator is close to the original one. This yields the simultaneous approximation of the operator and the almost norm-attaining function
	required by the Bollobás theorem.
	
	\begin{proof}[Proof of Theorem A] Fix $0 < \eps < 1$. Let us define all the parameters we need in this proof, which are quite a few as one might guess. Choose some number $1/2 < r < 1$. Choose also $\kappa > 0$ so small  that 
		\begin{equation} \label{eq12}
			\frac{\sqrt{2 \kappa}}{1-\sqrt{r}} + \frac{2 \kappa}{1-r} < \frac{\eps}{4}.
		\end{equation}
		Choose $0 < a < 1$ close to 1 so that 
		\begin{equation} \label{eq13}
			1 - a < \min \left\{ \eps, \frac{\kappa}{3} \right\}.
		\end{equation}
		Then choose numbers $b, c \in \R$ with 
		\begin{equation} \label{eq14}
			a < b < c < 1.
		\end{equation}
		Choose now $\tau > 0$ so small  that 
		\begin{equation} \label{eq15}
			\tau < \min \left\{ \frac{\kappa}{3}, \frac{\eps}{16} \right\}.
		\end{equation}
		Finally, let us choose $\eta > 0$ small enough so that 
		\begin{equation} \label{eq16}
			\eta \left( 1 + \frac{1}{1-c} \right) < \frac{\kappa}{3}
		\end{equation}
		and 
		\begin{equation} \label{eq17}
			\frac{\eta}{1-c} + \tau < \frac{\eps}{8}.
		\end{equation}
		This $\eta$ will be the one which will give us the result.

		Now we can finally start the proof. Let us take $T_0 \in \mathcal{L}(C(K), C(S))$ with $\|T_0\| = 1$ and $f_0 \in C(K)$ with $\|f_0\| = 1$ to be such that $\|T_0(f_0)\| > 1 - \eta$. By the representation theorem for operators from $C(K)$ into $C(S)$, there exists a $w^*$-continuous map $\mu_0: S \rightarrow \mathcal{M}(K)$ such that 
		\begin{equation*}
			T_0 f (s) = \int_K f d \mu_0(s) 
		\end{equation*}
		for every $f \in C(K)$ and every $s \in S$. Also, $\|\mu_0\| = \| T_0 \| = 1$. Since $T_0 f_0 \in C(S)$ and its norm is larger than $1-\eta$, there exists $s_0 \in S$ such that 
		\begin{equation*}
			\left| \int_K f_0 d \mu_0(s_0) \right| > 1 - \eta.
		\end{equation*}
		Now, take $\lambda \in \T$ to be such that, after setting $\mu:= \lambda \mu_0$, one gets 
		\begin{equation*}
			\re \int_K f_0 d \mu(s_0) > 1 - \eta
		\end{equation*}
		and $\|\mu\| = 1$. Let us apply Lemma \ref{lemma2} to this $\mu$, to $f_0$, to $s_0$ and to the parameters we have picked before, that is, to $a, b, c, \tau$ and consider the $\alpha$ from Lemma \ref{lemma2} to be $\alpha := \eta$. We then obtain $h \in C(K)$ with $\|h\| = 1$, a closed set $A \subseteq K$, a nonempty open set $U \subseteq S$ and a $w^*$-continuous map $\mu_1: S \rightarrow \mathcal{M}(K)$ such that 
		\begin{equation} \label{eq18}
			\|h - f_0\| \leq 1 - a \stackrel{(\ref{eq13})}{<} \eps,
		\end{equation}
		$|h(t)| = 1$ for every $t \in K \setminus A$, the set $K \setminus A$ is nonempty, $|\mu_1(s)|(A) = 0$ for every $s \in U$, $\|\mu_1\| \leq 1$,
		\begin{equation} \label{eq19}
			\|\mu_1 - \mu\| < m := \frac{\eta}{1 - c} + \tau 
		\end{equation}
		and 
		\begin{equation} \label{eq20}
			\re \int_K h d \mu_1(s) > 1 - d, \forall s \in U
		\end{equation}
		where $d:= (1 - a) + \eta + m$. From (\ref{eq13}), (\ref{eq16}) and (\ref{eq15}), we have 
		\begin{equation} \label{eq21}
			d = (1 - a) + \eta + \frac{\eta}{1-c} + \tau < \frac{\kappa}{3} + \frac{\kappa}{3} + \frac{\kappa}{3} = \kappa.
		\end{equation}
		Also, by (\ref{eq17}), we have that 
		\begin{equation} \label{eq22}
			m = \frac{\eta}{1-c} + \tau < \frac{\eps}{8}.
		\end{equation}

		Now, since $\|\mu_1\| \leq 1$, (\ref{eq20}) implies that 
		\begin{equation*}
			\re \int_K h d \mu_1(s) > \|\mu_1\| - d, \forall s \in U.
		\end{equation*}

		We now iterate Lemma \ref{lemma3}. Let us do this carefully. Set $\mu^{(0)} := \mu_1$, $U_0 := U$ and $\delta_n:= r^n d$ for $n=0,1,2,3, \ldots$. Inductively, suppose that $\mu^{(n)}: S \rightarrow \mathcal{M}(K)$ is $w^*$-continuous, $U_n \subseteq S$ is nonempty and open, and 
		\begin{equation*}
			|\mu^{(n)}(s)|(A) = 0, \forall s \in U_n, \ \ \|\mu^{(n)}\|\leq 1 
		\end{equation*}
		and 
		\begin{equation*}
			\re \int_K h d \mu^{(n)}(s) > \|\mu^{(n)}\| - \delta_n, \forall s \in U_n.
		\end{equation*}
		Applying Lemma \ref{lemma3} to $\mu^{(n)}$, $U_n$, $A$, $h$ and $\delta_n$, we obtain a $w^*$-continuous map $\mu^{(n+1)}: S \rightarrow \mathcal{M}(K)$ and a nonempty open set $U_{n+1} \subseteq U_n \subseteq S$ such that 
		\begin{itemize}
			\item[(i)] $|\mu^{(n+1)}(s)|(A)=0$ for every $s \in U_{n+1}$,
			\item[(ii)] $\|\mu^{(n+1)}\| \leq \|\mu^{(n)}\| \leq 1$, 
			\item[(iii)] $\re \int_K h d \mu^{(n+1)}(s) > \| \mu^{(n+1)}\| - r \delta_n = \| \mu^{(n+1)}\| - \delta_{n+1}, \forall s \in U_{n+1}$ and 
			\item[(iv)] $\|\mu^{(n+1)} - \mu^{(n)}\| \leq \sqrt{2 \|\mu^{(n)}\| \delta_n} + 2 \delta_n \leq \sqrt{2 r^n d} + 2 r^n d$.
		\end{itemize}
		Then, the induction follows.

		First, notice that 
		\begin{equation*}
			\sum_{n=0}^{\infty} \|\mu^{(n+1)} - \mu^{(n)}\| \stackrel{(\text{iv})}{\leq} \sum_{n=0}^{\infty} ( \sqrt{2r^n d} + 2r^n d) \stackrel{(\ref{eq21})}{<} \sum_{n=0}^{\infty} (\sqrt{2 r^n \kappa} + 2 r^n \kappa).
		\end{equation*}
		The right hand side sum is finite as it is a geometric series. In fact, 
		\begin{equation*}
			\sum_{n=0}^{\infty} \sqrt{2 r^n \kappa} = \sqrt{2 \kappa} \sum_{n=0}^{\infty} r^{n/2} = \frac{\sqrt{2 \kappa}}{1 - \sqrt{r}}
		\end{equation*}
		and 
		\begin{equation*}
			\sum_{n=0}^{\infty} 2 r^n \kappa = 2 \kappa \sum_{n=0}^{\infty} r^n = \frac{2 \kappa}{1 - r}.
		\end{equation*}
		Thus, we have that
		\begin{equation*}
			\sum_{n=0}^{\infty} \| \mu^{(n+1)} - \mu^{(n)}\| < \frac{\sqrt{2 \kappa}}{1 - \sqrt{r}} + \frac{2 \kappa}{1 -r } \stackrel{(\ref{eq12})}{<} \frac{\eps}{4}.
		\end{equation*}
		This means that $(\mu^{(n)})_{n=0}^{\infty}$ is a Cauchy sequence in the Banach space of all $w^*$-continuous maps from $S$ into $\mathcal{M}(K)$ endowed with the norm $\| \nu \| = \sup_{s \in S} \| \nu(s)\|$. So, there exists a $w^*$-continuous map $\nu: S \rightarrow \mathcal{M}(K)$ with 
		\begin{equation} \label{eq23}
			\| \nu - \mu_1 \| < \frac{\eps}{4}.
		\end{equation}
		Now, notice that 
		\begin{equation} \label{eq24}
			\|\nu - \mu\| \leq \|\nu - \mu_1\| + \|\mu_1 - \mu\| \stackrel{(\ref{eq23})}{<} \frac{\eps}{4} + m \stackrel{(\ref{eq22})}{<} \frac{\eps}{4} + \frac{\eps}{8} = \frac{3 \eps}{8}.
		\end{equation}
		Notice also that, from (\ref{eq24}) and the fact that $\|\mu\| = 1$, we have that $\nu$ is non-zero. We will now conclude the proof by proving the following two claims.
		
		\vspace{0.2cm}
		\noindent
		{\it Claim 1}: The operator represented by $\nu$ attains its norm at $h$.
		
		\begin{proof} For each $n$, let us choose $s_n \in U_n$. Then, 
			\begin{equation*}
				\re \int_K h d \mu^{(n)}(s_n) > \| \mu^{(n)}\| - \delta_n.
			\end{equation*}
			Thus, 
			\begin{eqnarray*}
				\sup_{s \in S} \left| \int_K h d \nu(s) \right| &\geq& \re \int_K h d \nu (s_n) \\
				&\geq& \re \int_K h d \mu^{(n)}(s_n) - \| \nu - \mu^{(n)}\| \\
				&>& \| \mu^{(n)}\| - \delta_n - \| \nu - \mu^{(n)}\|.
			\end{eqnarray*}
			Letting $n \rightarrow \infty$ and using that $\delta_n \rightarrow 0$ and $\mu^{(n)} \rightarrow \nu$ in norm as $n \rightarrow \infty$, we get 
			\begin{equation*}
				\sup_{s \in S} \left| \int_K h d \nu(s) \right| \geq \| \nu \|.
			\end{equation*}
			On the other hand, since $h \in S_{C(K)}$, we have that 
			\begin{equation*}
				\left| \int_K h d \nu(s) \right| \leq \| \nu(s) \| \leq \| \nu\| 
			\end{equation*}
			for every $s \in S$. This shows that 
			\begin{equation} \label{eq25}
				\sup_{s \in S} \left| \int_K h d \nu(s) \right| = \|\nu\|.
			\end{equation}
			Recall that $\lambda \in \T$ is the one such that $\mu = \lambda \mu_0$. Now, let us define $\tilde{\mu}: S \rightarrow \mathcal{M}(K)$ by 
			\begin{equation*}
				\tilde{\mu}(s):= \overline{\lambda} \cdot \frac{\nu(s)}{\|\nu\|}, \forall s \in S.
			\end{equation*}
			Since $\nu$ is $w^*$-continuous, so is $\tilde{\mu}$. Also,
			\begin{equation*}
				\| \tilde{\mu} \| = \sup_{s \in S} \| \tilde{\mu}(s) \| = \frac{1}{\| \nu\|} \sup_{s \in S} \| \nu(s)\| = 1.
			\end{equation*}
			Now, define $T \in \mathcal{L}(C(K), C(S))$ to be the operator represented by $\tilde{\mu}$, that is,
			\begin{equation*}
				Tf(s) = \int_K f d \tilde{\mu}(s) 
			\end{equation*}
			for every $f \in C(K)$ and $s \in S$. Then, $\|T\| = \|\tilde{\mu}\| = 1$. Also, by (\ref{eq25}), we have that 
			\begin{equation*}
				\|T(h)\| = \frac{1}{\|\nu\|} \sup_{s \in S} \left| \int_K h d \nu(s) \right| = 1. \hspace{0.5cm} (\#_3)
			\end{equation*}
		\end{proof}

		\vspace{0.2cm}
		\noindent
		{\it Claim 2}: $\|T - T_0\| < \eps$.
		
		\begin{proof} By the representation theorem, we have that 
			\begin{eqnarray*}
				\|T - T_0\| = \| \tilde{\mu} - \mu_0\| &=& \left\| \overline{\lambda} \cdot \frac{\nu}{\|\nu\|} - \mu_0 \right\| \\
				&=& \left\| \frac{\nu}{\|\nu\|} - \lambda \mu_0 \right\| \\
				&\leq& \left\| \frac{\nu}{\|\nu\|} - \nu \right\| + \| \nu - \mu\| \\
				&=& | \| \nu\| - \|\mu\| | + \| \nu - \mu \|  
				\leq 2 \| \nu - \mu \| 
				\stackrel{(\ref{eq24})}{<} \frac{3 \eps}{4} < \eps. \hspace{0.5cm} (\#_4)
			\end{eqnarray*}
		\end{proof}

		We already have by (\ref{eq18}) that $\|h - f_0\| < \eps$. Since $T$ attains its norm at $h$ by ($\#_3$) and $\|T - T_0\| < \eps $ by ($\#_4$), we are done.
	\end{proof}

	Let us also observe that all the perturbations used in the proof of Theorem A preserve compactness. Indeed, if $T_0$ is compact, then the operator represented by $\mu=\lambda\mu_0$ is the compact operator $\lambda T_0$. On the other hand, the operator $T_1$ represented by the map $\mu_1$ from Lemma \ref{lemma2} satisfies
	\begin{equation*}
		[T_1f](s) = \lambda[T_0f](s) - \lambda\psi(s)[T_0(fG)](s),
	\end{equation*}
	and hence it is compact. Moreover, if the operator $T_n$ represented by $\mu^{(n)}$ is compact, then in the first case of Lemma \ref{lemma3}, the operator $T_{n+1}$ is obtained from $T_n$ by adding a rank-one operator, while in the second case it satisfies 
	\begin{equation*}
		[T_{n+1}f](s) = (1-\psi(s))[T_nf](s) + \psi(s)[T_n(fq)](s).
	\end{equation*}
	Thus, all the operators arising in the iterative construction are compact. In the proof of Theorem A, the limiting operator is also compact, since the space of compact operators is closed in the operator norm. Finally, multiplying an operator by a scalar and normalizing it do not affect compactness. Therefore, if the initial operator $T_0$ is compact, then the proof of Theorem A produces a compact operator $T$. We have the following corollary.

	\begin{corollary} \label{compact-case} Let $K$ and $S$ be compact Hausdorff spaces. For every $0<\eps<1$, there exists $\eta(\eps)>0$ such that, whenever $T_0\in S_{\mathcal K(C(K),C(S))}$ and $f_0\in S_{C(K)}$ satisfy $\|T_0f_0\|>1-\eta(\eps)$ there exist $T\in S_{\mathcal K(C(K),C(S))}$ and $h\in S_{C(K)}$ such that
		\begin{equation*}
			\|Th\|=1, \ \ \ \|T-T_0\|<\eps \ \ \ \mbox{and} \ \ \ \|h-f_0\|<\eps.
		\end{equation*}
	\end{corollary}
	
	\noindent 
	\textbf{Acknowledgements}: The authors would like to thank Miguel Martín for several fruitful conversations on the topic of this manuscript.
	
	\noindent 
	\textbf{Funding}: S. Dantas has been supported by the grants PID2021-122126NB-C31 and PID2021-122126NB-C33 funded by MICIU/AEI/ 10.13039/ 501100011033 and by ERDF/EU. Helena del Río was supported by Grant PRE2022-103590, funded by MICIU/AEI/ 10.13039/ 501100011033 and by ESF+, Grant PID2021-122126NB-C31, funded by MICIU/AEI/10.13039/ 501100011033 and by ERDF/EU, by the "María de Maeztu" Excellence Unit IMAG, funded by MICIU/AEI/10.13039/501100011033 under reference CEX2020-001105-M and by Junta de Andaluc\'ia, grant FQM-0185.

\end{document}